\documentclass[11pt,reqno]{amsart}
\usepackage{graphicx,color}

\usepackage{amsmath,amsfonts,amssymb,amsthm,mathrsfs,amsopn}
\usepackage{latexsym}

\theoremstyle{plain}

\theoremstyle{plain}
\newtheorem{theorem}{Theorem} [section]

\newtheorem{lemma}[theorem]{Lemma}

\theoremstyle{definition}
\newtheorem{definition}[theorem]{Definition}

\theoremstyle{remark}

\numberwithin{equation}{section}

\numberwithin{theorem}{section}
\numberwithin{equation}{section}
\numberwithin{figure}{section}

\def\mean#1{\mathchoice%                                                             %integral mean
         {\mathop{\kern 0.2em\vrule width 0.6em height 0.69678ex depth -0.58065ex
                 \kern -0.8em \intop}\nolimits_{\kern -0.4em#1}}%
         {\mathop{\kern 0.1em\vrule width 0.5em height 0.69678ex depth -0.60387ex
                 \kern -0.6em \intop}\nolimits_{#1}}%
         {\mathop{\kern 0.1em\vrule width 0.5em height 0.69678ex
             depth -0.60387ex
                 \kern -0.6em \intop}\nolimits_{#1}}%
         {\mathop{\kern 0.1em\vrule width 0.5em height 0.69678ex depth -0.60387ex
                 \kern -0.6em \intop}\nolimits_{#1}}}

\def\N{\mathbb N}

\def\R{\mathbb R}

\def\eps{\varepsilon}

\def\1{\mathbf 1}

\def\H{\mathcal H}

\def\F{\mathcal F}

 \DeclareMathOperator{\dist}{dist}

\title[Singular sets in free interface problems]{A note on the dimension of the singular set\\ in free interface problems}

\author[G.~De Philippis]{Guido De Philippis}
\address{Institut f\"ur Mathematik, Universit\"at Z\"urich --
CH-8057 Z\"urich}
\email{guido.dephilippis@math.uzh.ch}

\author[A. Figalli]{Alessio Figalli}
\address{Department of Mathematics,
The University of Texas at Austin, 1 University Station C1200,
Austin TX 78712, USA}
\email{figalli@math.utexas.edu}

\begin{document}

\begin{abstract}
The aim of this note is to investigate the size of the singular set of a general class of 
free interface problems. We show porosity of the singular set, obtaining as a corollary that
both its Hausdorff and Minkowski dimensions are strictly smaller than $n-1$.
\end{abstract}

\maketitle

\section{Introduction}
Let \(\Omega\subset \R^n\) be an open set. Given  \(E\subset \Omega\) and \(u\in W^{1,2}(\Omega)\) we define 
\begin{equation}\label{defF}
\mathcal F^{\alpha,\beta}_\eps (u,E):=\varepsilon\, P(E,\Omega)+\int_{\Omega} a_E(x)|\nabla u|^2,
\end{equation}
where  \(a_E(x):=\beta \1_E(x)+\alpha \1_{\Omega\setminus E}(x)\),  \(\varepsilon\in (0,1)\) and \(0<\alpha<\beta<\infty\) are given constants, and $P(E,\Omega)$ denotes the perimeter of $E$ relative to $\Omega$.
We are interested in the regularity of \((\Lambda,r_0)\)-minimizers of \(\F\) in \(\Omega\), namely in couples \((u,E)\) such that
\begin{equation}\label{lambda}
\F^{\alpha,\beta}_\eps(u, E)\le \F^{\alpha,\beta}_\eps (v,E)+\Lambda|E\Delta F|
\end{equation}
for all \(F\subset \Omega\), \(v\in W^{1,2}(\Omega)\)  such that \(E\Delta F\Subset B_{x,r}\Subset \Omega\), \(u=v\) on \(\Omega\setminus B_{x,r}\), \(r\le r_0\). Minimizers and \((\Lambda, r_0)\)-minimizers of \(\mathcal F_\eps^{\alpha,\beta}\) naturally arise in several problems
from material sciences, see \cite{AB, EF, F, FK} and references therein.

\medskip
In \cite{F}  it has been established that if \((u,E)\) is \((\Lambda,r_0\))-minimizer, then \(\partial E\) is regular outside a relatively closed set of vanishing \(\mathcal H^{n-1}\) measure (here and in the sequel,
$\H^{n-1}$ denotes the $(n-1)$-dimensional Hausdorff measure). More precisely if we define the regular set 
\begin{equation}\label{defreg}
{\rm  Reg} (E):=\big\{x\in \partial E\cap \Omega: \textrm{ \(\partial E\) is a \(C^{1,\gamma}\) hypersurface in a neighborhood  of \(x\) for some \(\gamma\in (0,1)\)} \big\}\,
 \end{equation}
and the singular set 
\begin{equation}\label{defsing}
\Sigma(E):=(\partial E\cap \Omega) \setminus {\rm  Reg} (E)\,,
\end{equation}
then \(\H^{n-1}(\Sigma (E))=0\), see Section \ref{partial} below for a more detailed discussion. 
%\footnote{Actually the theory in \cite{} has been done only for plain minimizers (i.e for\((0,\infty)\)-minimizer with our notations) however, since the term \(|E\Delta F|\) appearing in \eqref{lambda} is of higher order with respect to the perimeter, it can be repeat verbatim to obtain  regularity also in the case of \((\Lambda,r_0)\)-minimizers.}. 
On the other hand  nothing is known concerning  the Hausdorff dimension of \(\Sigma (E)\). In this note we will address this issue by proving the following:

\begin{theorem}\label{thm:main1} There is a constant \(\kappa=\kappa(n,\beta/\alpha)>0\) such that,  for every   \((\Lambda,r_0)\)-minimizer of \(\F^{\alpha,\beta}_\eps\),
\begin{equation}\label{eq:haus}
\dim_{\H} \Sigma (E)\le n-1-\kappa.
\end{equation}
\end{theorem}
Note that \(\kappa\) depends only on $n$ and  \(\beta/\alpha\) but not on \(\eps\), see also the comments after Theorem \ref{thm:main2} below.

\medskip

A well-known classical strategy to study the dimension of singular sets in geometric   problem is the study of blow-ups of minimizers around a singular point. If one is able to classify the singularities of blow-ups then, applying the so-called Federer dimension reduction argument (see for instance \cite[Appendix]{Si}), one can obtain estimates on the dimension of the singular set of a generic minimizers. In order to apply this strategy one needs to have some monotonicity formula at hand. Indeed, roughly speaking, a monotonicity (or almost monotonicity) formula allows one to  classify blow-ups limits in a sufficiently precise way to understand the dimension of their singularities. For minimizers of \eqref{defF} monotonicity formulas are known only under some  very restrictive assumptions (see \cite[Lemma 3.1]{F}) and thus are not suitable to study blows limit. To prove Theorem \ref{thm:main1} we will then follow a different route, namely we will show that \(\Sigma (E)\) is $\sigma$-\emph{porous} in \(\partial E\) for some \(\sigma=\sigma(n,\beta/\alpha)>0\), see Definition \ref{def:porosity}. From this fact Theorem \ref{thm:main1} will follow by classical results in measure theory, see Lemma \ref{poro}. In particular Theorem \ref{thm:main1} will be a consequence of Theorem \ref{thm:main2} below. Moreover, by using Lemma \ref{poro2} one can actually provide an estimate on the Minkowski content of \(\Sigma(E)\).

To explain Theorem \ref{thm:main2}  we need the following definition:

\begin{definition}\label{def:porosity}Given \(\Sigma\subset K\subset B_1\),  we say that \(\Sigma\) is \((\sigma, \hat \varrho)\)-porous in \(K\) if the following holds: For every \(x\in K\) and every \(\varrho \le \hat \varrho\), there exists \(y\in K\cap B_{x,\varrho}\) and \(r \in (\sigma \varrho , \varrho) \) such that
\begin{equation*}
B_{y,r}\cap K \subset K\setminus \Sigma.	
\end{equation*}
\end{definition}
Let us also introduce the following notation, which will be useful in the sequel. Given \((u,E)\)  a \((\Lambda,r_0)\)-minimizer of \(\F\) in \(B_1\), \(x\in B_1\), and \(\varrho\le \dist(x,\partial B_1)\), we define the normalized Dirichlet energy of \(u\) as
\begin{equation}\label{eq:dir}
D_u (x,\varrho):=\frac{1}{\varrho^{n-1}} \int_{B_{x,\varrho}} | \nabla u|^2.
\end{equation}
In case $x$ is the origin we will simply write \(D_u(\varrho)\). We can now state the  main result of this paper

\begin{theorem}\label{thm:main2} There is a positive constant \(\sigma=\sigma(n,\beta/\alpha)\) such that  the following holds: For  every  \((\Lambda,r_0)\)-minimizer of \(\F^{\alpha,\beta}_\eps \) in \(B_1\) there exists a radius \(\hat \varrho=\hat \varrho\big(n,r_0,\,\Lambda/\eps, \,\alpha D_u(1)/\eps \big)>0 \) such that \(\Sigma(E)\cap \overline B_{1/2}\) is \((\sigma,\hat \varrho)\)-porous in \(\partial E\cap B_1\).
\end{theorem}

Note that \(\sigma\) depends only on $n$ and \(\beta/\alpha\). This is crucial in showing that the constant \(\kappa\) appearing in Theorem \ref{thm:main1} depends only on $n$ and \(\beta/\alpha\) as well. On the other hand the redius  \(\hat \varrho\)  shall be thought as a \emph{regularity scale}, i.e., as the scale at which the perimeter term becomes dominant. In this respect the fact that it depends also on  $\eps$, $\Lambda$, and \(D_u(1)\) is quite natural.
This can be seen for instance by looking at the asymptotic behavior of the family of minimizers of the following problems as \(\eps \to 0\):
\begin{equation}\label{Peps}\tag{\(P_\eps\)}
\min\Big\{ \F^{\alpha, \beta}_\eps(u,E):\qquad |E|=|\Omega|/2\quad u=u_0\textrm{ on \(\partial \Omega\)} \Big\}
\end{equation}
where $\Omega=[0,1]^2$ is the unit square in $\R^2$ and \(u_0=x_1\).
In this case minimizers exhibit finer and finer microstructures, and the gradient jumps along a finer and finer family of curves which propagate in the direction $e_2$

We conclude this first section by recalling that the use of porosity in the study of the dimension of the singular set of minimizers of variational problems when no monotonicity formulas are available   already appeared in \cite{KM} for minimizers of quasi-convex functionals, and in   \cite{Dms, MSms, Rms} concerning the Mumford-Shah functional. In particular, by using the porosity of the singular set of minimizers of the Mumford-Shah functional, the authors have been able to prove in \cite{DF} a  higher integrability property of the gradients of the minimizers conjectured by De Giorgi in the 90's.

\medskip

The proof of Theorem \ref{thm:main2} is based on the following idea which we believe can be applied also to different types of problems.

First, to remove the dependence on $\eps$ we notice that
by scaling one can reduce itself to the case $\eps=1$,
the price to pay being that the size of the Dirichlet energy increases by  a factor $1/\eps$.
Then, by a comparison argument we show that either the Dirichlet energy is below a fixed threshold $C_0$ or it decays geometrically (see Lemma \ref{alt}). By scaling this implies that, below a suitable scale
that can depend on $\eps$,
one has reduced itself to the case $\eps=1$ and $D_u(x,\rho)\leq C_1$ (see Lemma \ref{energybound}),
hence removing the dependence on $\eps$.

We then prove the porosity result. For this we observe that, by known excess-type regularity theorems (see Theorem \ref{eps0} and \cite[Section 5]{F})
it follows that ${\rm Reg}(E)$ is an open relatively to $\partial E$
and $\H^{n-1}(\Sigma(E))=0$.
This implies in particular that, given a \((\Lambda,r_0)\)-minimizer in $B_1$ with
$\eps=1$ and $D_u(1)\leq C_1$, there exists a ball inside $B_{1/2}$ where $\partial E$ is regular.
Our observation is that, using a simple contradiction argument 
based on the compactness of \((\Lambda,r_0)\)-minimizers, the radius of this ``regularity ball'' is universal. This fact combined with the fact the  \((\Lambda,r_0)\)-minimizers are invariant under scaling
allows to transfer this information inside any ball and prove that, inside any ball $B_\rho(x)$,
there exists a ball with comparable radius where $\partial E$ is smooth
(see the proof of Theorem \ref{thm:main2} in Section \ref{3}).
This concludes the proof of the porosity of the singular set.

It is worth pointing out that this argument is robust enough that can be applied also 
to the family of anisotropic energies considered in \cite{FK}.

\medskip
This paper is organized as follows:  in Section \ref{2} we summarize some known results concerning minimizers \((\Lambda,r_0)\)-minimizers and we prove some preliminary lemmas with a particular attention in underlining the dependence of the constants on the parameters. Then, in Section \ref{3} we provide the proofs of Theorems \ref{thm:main1} and \ref{thm:main2}.

\medskip

After the writing of this paper was completed, we learned that Fusco and Julin \cite{FJ} had just 
obtained similar results with related but somehow different techniques.

\section{Preliminaries and technical lemmas.}\label{2}
In this section we prove some technical lemmas that we will need in the sequel. 
\subsection{Scaling}\label{subsection:scaling} Let \((u,E)\) be a \((\Lambda,r_0)\)-minimizer of \(\F^{\alpha,\beta}_\eps\) in \(B_1\). Then:

\medskip
\noindent
\(\bullet\)\emph{Horizontal scaling:} For every \(x\in B_{1}\) and \(r \le \dist(x,\partial B_1)\), let us define  \(u^{x,r}(y):=r^{-1/2}u(x+ry)\) and \(E^{x,r}:=(E-x)/r\). Then \((u^{x,r}, E^{x,r})\) is a \((\Lambda r, r_0/r)\)-minimizer of \(\F^{\alpha,\beta}_\eps\) in \(B_1\). Note also that
\begin{equation}\label{dirhor}
D_{u}(x,\varrho)=D_{u^{x,r}}(0,\varrho/r).
\end{equation}

\medskip
\noindent
\(\bullet\)\emph{Vertical scaling:} For every \(\mu>0\), define \(v(x):=\sqrt{\mu} u(x)\). Then \(( v,E)\) is  a \(( \Lambda/ \eps , r_0)\)-minimizer of \(\F_{1}^{\alpha/\mu\eps , \beta /\mu\eps}\) and 
\begin{equation}\label{dirvert}
D_{v}(x,\varrho)=\mu D_{u}(x,\varrho).
\end{equation}
%In particular,  up to substitute \(u\) with \(\sqrt{\alpha/\eps}\, u^{0,\varrho_0}\)  and \(E\) with \(E^{0,\varrho_0}\)  with 
%\(\varrho_0=\min\{\eps, r_0,1\}\),  in the sequel we will always work under the following assumption:   
%
%\noindent
%\begin{equation}\label{A}\tag{A} 
%\textrm{\((u,E)\) is a \((\Lambda,1)\) minimizer of \(\F^{1,\beta/\alpha}_1\) in \(B_1\).}
%\end{equation}
%
% Note that under this transformation
%\begin{equation}\label{dirscala1}
%D_v(1)=\frac{\alpha}{\eps} D_u(0,\varrho_0)\le C\big(\alpha, \eps, r_0, D_u(1)\big). 
%\end{equation}

\subsection{Upper density estimates}
\label{subsection:density}
Let \((u,E)\) be a    \((\Lambda, r_0)\)-minimizer of \(\F^{1,\beta/\alpha}_1\) in \(B_1\).
Testing the minimality of \((u,E)\) with \((u, E\setminus B_{x,\varrho})\) we obtain
\begin{equation}\label{eq:upperdensity}
P(E,B_{x,\varrho})+\frac{\beta-\alpha}{\alpha} \int_{ E\cap B_{x,\varrho} }|\nabla u|^2\le n\omega_n \varrho^{n-1}+\Lambda \varrho^n \qquad \forall \, x\in B_{1}, \quad \varrho \le  \min\{\dist(x,\partial B_1),r_0\}.
\end{equation}

\subsection{The equation for  $u$ and some consequences} Let \((u,E)\) be a    \((\Lambda, r_0)\)-minimizer of \(\F^{1,\beta/\alpha}_1\) in \(B_1\).
If we test the minimality of \((u,E)\) with \((u+ \eta \varphi, E)\), \(\varphi\in W^{1,2}_{0}(B_{x,\varrho})\), and we let \(\eta \to 0\),  we get
\begin{equation}\label{eq:equ}
\beta \int_{E\cap B_{x,\varrho}} \nabla u\cdot \nabla \varphi+ \alpha \int_{ B_{x,\varrho}\setminus E} \nabla u\cdot \nabla \varphi=0\qquad \forall \, \varphi\in W^{1,2}_{0}(B_{x,\varrho}).
\end{equation}
In particular, if \(v\) denotes the harmonic function with the some boundary data of \(u\) on \(\partial B_{x,r}\),
plugging \(\varphi:=u-v\) in \eqref{eq:equ}  and using that 
\[
 \int_{ B_{x,\varrho}} \nabla v\cdot (\nabla u-\nabla v)=0,
\]
we obtain 
\begin{equation}\label{eq:harcomp}
\int_{B_{x,\varrho}} |\nabla u-\nabla v|^2\le \frac{(\beta-\alpha)^2}{\alpha^2}\int_{B_{x,\varrho}\cap E} |\nabla u|^2.
\end{equation}
We can now prove the following Lemma, see also  \cite[Lemma 2.2]{F}.

\begin{lemma}\label{alt}
Let \((u,E)\) be a    \((\Lambda, r_0)\)-minimizer of \(\F^{1,\beta/\alpha}_1\) in \(B_1\).
There exists a constant  \(C_0=C_0(n,\beta/\alpha)>0\) such that 
\begin{align}
\label{bound}
\textrm{either}&&&  D_{u}(x,\varrho)\le  C_0\\
\label{decade}
\textrm{or}&&&D_{u}(x, \varrho/16 )\le  \frac{1}{4} D_{u}(x, \varrho)
\end{align}
for every \(x\in B_{1}\) and \(\varrho\le\min\{\dist(x,\partial B_1),r_0,1/\Lambda\}\).
\end{lemma}

\begin{proof}Let \(C_0\gg1\) to be fixed and assume that for some \(x\in B_{1}\) and \(\varrho\le\min\{\dist(x,\partial B_1),r_0,1/\Lambda\}\) we have
\[
 \int_{B_{x,\varrho}} |\nabla u|^2\ge  C_0\varrho^{n-1},
\]
so that \eqref{bound} fails. By the above equation, \eqref{eq:upperdensity}, and using that  \(\Lambda \varrho\le 1 \), we get 
\[
\frac{\beta-\alpha}{\alpha}\int_{B_{x,\varrho}\cap E}|\nabla u|^2\le \frac{n\omega_n+1}{C_0} \int_{B_{x,\varrho}} |\nabla u|^2,
\]
that combined with \eqref{eq:harcomp} gives
\begin{equation}
\label{eq:uv close}
\int_{B_{x,\varrho}} |\nabla u-\nabla v|^2\le \frac{C(n,\beta/\alpha)}{C_0} \int_{B_{x,\varrho}} |\nabla u|^2,
\end{equation}
where \(v\) is the harmonic function with the same boundary data of \(u\) on \(\partial B_{x,\varrho}\).
We notice that as a consequence of the harmonicity of $v$ the function $|\nabla v|^2$ is subharmonic, hence
$$
\frac1{r^n}\int_{B_{x,r}}|\nabla v|^2 \leq \frac1{\varrho^n} \int_{B_{x,\varrho}}|\nabla v|^2 \qquad \forall\,r \in (0, \varrho).
$$
Applying this inequality with $r=\varrho/\lambda$ with $\lambda \geq 1$,
together with \eqref{eq:uv close} and the fact that $|\nabla u|^2\leq 2|\nabla v|^2+2|\nabla u-\nabla v|^2$, we deduce that
\begin{equation}\label{treno3}
\begin{split}
\frac{\lambda ^{n-1}}{\varrho^{n-1}}\int_{B_{x,\varrho/\lambda}}|\nabla u|^2&\le \frac{2 \lambda ^{n-1}}{\varrho^{n-1}} \int_{B_{x,\varrho/\lambda}} |\nabla u-\nabla v|^2+\frac{2\lambda ^{n-1}}{\varrho^{n-1}} \int_{B_{x,\varrho/\lambda}}|\nabla v|^2\\
&\le \frac{2\lambda ^{n-1}}{\varrho^{n-1}} \int_{B_{x,\varrho}} |\nabla u-\nabla v|^2+\frac{2}{\lambda \varrho^{n-1}} \int_{B_{x,\varrho}}|\nabla v|^2\\
 &\le  \frac{C(n,\beta/\alpha)\lambda^{n-1}}{C_0\varrho^{n-1}} \int_{B_{x,\varrho}} |\nabla u|^2+\frac{2}{\lambda \varrho^{n-1}} \int_{B_{x,\varrho}}|\nabla u|^2,\\
\end{split}
\end{equation}
where in the last inequality we have also used that \(\int_{B_{x,\varrho}} |\nabla v|^2\le \int_{B_{x,\varrho}} |\nabla u|^2\) since \(v\) is harmonic.  Choosing \(\lambda=16\)  we have
\[
\frac{2}{\lambda}=\frac{1} {2\lambda^{1/2}}\,,
\]
hence, if  \(C_0\gg1\) is sufficiently big to ensure that
\[
 \frac{C(n,\beta/\alpha)\lambda^{n-1}}{C_0}\le \frac{1} {2\lambda^{1/2}},
\]
\eqref{treno3} follows from \eqref{decade}.
\end{proof}

The above lemma allows us to show that below a certain scale (which depends only on the total energy \(D_u(1)\)) the normalized energy \(D_{u} (x,\varrho)\) is bounded only in terms of $n$ and \(\beta/\alpha\). This fact will be crucial in  showing that the constant \(\sigma\) appearing in Theorem \ref{thm:main2} (as well as the constant \(\kappa\) in Theorem \ref{thm:main1}) depends only on $n$ and \(\beta/\alpha\). 

\begin{lemma}\label{energybound}
Let \((u,E)\) be a    \((\Lambda, r_0)\)-minimizer of \(\F^{1,\beta/\alpha}_1\) in \(B_1\).
There exist a constant  \(C_1=C_1(n,\beta/\alpha)>0\) and a radius \(\varrho_1=\varrho_1(n,r_0,\Lambda, D_u(1))>0\) such that
\begin{equation}\label{universalbound}
D_{u} (x,\varrho )\le C_1\qquad \forall \,x\  \in B_{1/2}, \quad \forall\,\varrho \le \varrho_1. 
\end{equation}
\end{lemma}

\begin{proof}
By continuity it is enough to prove \eqref{universalbound} at almost every point in \(B_{1/2}\).  Let \(C_1\gg1\) to be fixed, and for every  Lebesgue point \(x\in B_{1/2}\)  of \(|\nabla u|^2\) we define
\[
\varrho(x):=\sup\big\{ \varrho\in (0,1/2): D_u(x,\varrho)\le C_1\big\}.
\] 
Since \(x\) is a Lebesgue point for \(|\nabla u|^2\),
\[
\lim_{\varrho \to 0} D_u(x,\varrho)=\lim_{\varrho\to 0} \varrho \mean{B_{x,\varrho}} |\nabla u|^2=0.
\]
Hence \(\varrho(x)>0\) and \eqref{universalbound} will follow if we can show that \(\varrho(x)\ge \varrho_1\) for some \(\varrho_1=\varrho_1(n, r_0,\Lambda, D_u(1))>0\).  We  claim that if \(C_1\) is  sufficiently big, depending only on $n$ and \(\beta/\alpha\) , then for every \(x\in B_{1/2}\) and \(k\in \N\) such that 
\begin{equation}\label{treno1}
16^k\varrho(x)\le\min\{1/2,r_0,1/\Lambda\}
\end{equation}
we have
\begin{equation}\label{cresce}
D_u(x, 16^k \varrho(x))\ge C_1 4^k.
\end{equation}
We will prove \eqref{cresce} by induction on \(k\), the case \(k=0\) being trivial. 
To prove the induction step we first notice as a preliminary observation that
\begin{equation}
\label{eq:Du scale}
D_u(x,r) = \frac{1}{r^{n-1}} \int_{B_{x,r}}|\nabla u|^2 \leq  \frac{1}{r^{n-1}} \int_{B_{x,\varrho}}|\nabla u|^2
= \frac{\varrho^{n-1}}{r^{n-1}} D_u(x,\varrho)\qquad \forall\, r\in (0, \varrho).
\end{equation}
We now assume that \eqref{treno1} holds for \(k\) and that \eqref{cresce} holds for \(k-1\).
Then, if \(C_1\ge  C_0 16^{n-1}\) where \(C_0\) is the constant in Lemma \ref{alt},
applying \eqref{eq:Du scale} with $r=16^{k-1} \varrho(x)$ and $\varrho=16^{k} \varrho(x)$ we deduce that
\[
D_u(x, 16^k \varrho(x))\ge 16 ^{-(n-1)} D_u(x, 16^{k-1} \varrho(x))\ge 16^{-(n-1) }C_1 4^k \ge 16^{-(n-1) }C_1\ge C_0.
\]
Hence, we can apply Lemma \ref{alt} and the inductive step to infer that
\[
D_u(x, 16^k \varrho(x))\ge 4 D_u(x, 16^{k-1} \varrho(x))\ge C_1 4^k\,, 
\]
proving \eqref{cresce} for \(k\). Let now \(k_0=k_0(x)\) be the first \(k\) such that \eqref{treno1} fails for \(k=k_0+1\).
Then, 
since \(16^{k_0}\varrho(x)\ge\min\{1/2,r_0,1/\Lambda\}/16\),
according to \eqref{cresce} we get
\begin{equation}\label{treno2}
C_1 4^{k_0}\le D_u(x, 16^{k_0} \varrho(x))\le C(n, r_0,\Lambda) \,D_u(1)\,.
\end{equation}
This proves that \(k_0(x)\le k(n,r_0,\Lambda, D_u(1))\) therefore, by \eqref{treno1}, \(\varrho(x)\ge \varrho_1(n, r_0,\Lambda, D_u(1))>0\), as desired.
\end{proof}

\subsection{Lower density estimates} Section \ref{subsection:density} provides an upper-bound for \(P(E, B_{x,\varrho})\) in terms of \(\varrho^{n-1}\).
We now focus on the lower bound. In order to show that the constant \(\sigma\) appearing in the conclusion of Theorem \ref{thm:main2} depends only on $n$ and \(\beta/\alpha\) we need the lower bound on the density to depend only on this ratio. On the other hand the scale at which the density estimates become valid shall depends also on \(r_0\), \(\Lambda\) and  \(D_u(1)\).

\begin{lemma}\label{lowerdensity} 
Let \((u,E)\) be a    \((\Lambda, r_0)\)-minimizer of \(\F^{1,\beta/\alpha}_1\) in \(B_1\).
There exist a constant \(C_2=C_2(n,\beta/\alpha)>0\) and a radius \(\varrho_2=\varrho_2(n,r_0, \beta/\alpha, \Lambda, D_u(1))>0\) such that
\begin{equation}\label{eq:lowerdensity}
P(E,B_{x,\varrho})\ge  \varrho^{n-1}/C_2
\qquad \forall\, x\in B_{1/2}, \quad \varrho \le \varrho_2.
\end{equation}
\end{lemma}

\begin{proof} Let \(\varrho_1=\varrho_1(r_0,\Lambda, D_u(1))\) be the radius appearing in Lemma \ref{energybound} so that 
\[
D_u(x,\varrho)\le C_1
\]
for every \(x\in B_{1/2}\) and \(\varrho\le \varrho_1\),
with \(C_1=C_1(n,\beta/\alpha)\). If we consider \(v=u^{x,\varrho}\) and \(F=E^{x,\varrho}\) for \(x\in B_1\) and  \(\varrho\le \min\{\varrho_1, r_0, \delta/\Lambda)\) with \(\delta\ll 1\) to be fixed,  we see that \((v,F)\) is a \((\delta,1)\)-minimizer of \(\F_1^{1,\beta/\alpha}\) in \(B_1\). In addition
\begin{equation}\label{rubare2}
D_v(1)\le C_1(n,\beta/\alpha).
\end{equation}
Hence, if \(\delta\) is sufficiently small (depending only on \(C_1\)) we can argue as in \cite[Section 3]{F} to obtain the lower density estimates 
\begin{equation}\label{rubare}
P(F, B_{\varrho})\ge \varrho^{n-1}/\hat C \qquad \forall \, \varrho\le 1/2\,,
\end{equation}
see Lemmas 3.1 and 3.3 and the subsequent corollaries in \cite{F}. In particular, as it is clear from the proofs in \cite{F}, the  constant \(\hat C \) in \eqref{rubare} depends only on \(D_v(1)\), which in turn depends only on $n$ and \(\beta/\alpha\) (thanks to \eqref{rubare2}). Scaling back to \(E\), \eqref{rubare} implies \eqref{eq:lowerdensity}.
\end{proof}

A standard consequence of the lower density estimates is that \(\mathcal H^{n-1}(\partial E\setminus \partial^{*} E)=0\), where \(\partial^* E\) is the reduced boundary of \(E\), see \cite[Theorem 16.14]{M}. In particular, up to enlarge \(C_2\) and reduce \(\varrho_2\), combining Lemma \ref{lowerdensity} and Section \ref{subsection:density} we have that, if  \((u,E)\) be a    \((\Lambda, r_0)\)-minimizer of \(\F^{1,\beta/\alpha}_1\) in \(B_1\),
\begin{equation}\label{density}
\varrho^{n-1}/C_2\le \mathcal H^{n-1}(\partial E\cap B_{x,\varrho})\le C_2\varrho^{n-1} \qquad \forall\, x\in B_{3/4}, \quad \varrho \le \varrho_2,
\end{equation}
for some \(C_2=C_2(n,\beta/\alpha)\) and \(\varrho_2=\varrho_2(n,r_0, \beta/\alpha, \Lambda, D_u(1))\).

\subsection{The \(\eps\)-regularity theorem and convergence of minimizers}\label{partial} We recall the following theorem which has been proved in \cite{F} for \(\Lambda=0\), see Section 5 therein. Since, at small scales, volume terms are lower order with respect to surface terms, the proof can be repeated almost verbatim  for \((\Lambda,r_0)\)-minimizers.

\begin{theorem}\label{eps0}Let  \((u,E)\) be a \((\Lambda,1)\)-minimizer of \(\F^{1,\beta/\alpha}_1\) in \(B_{x,\varrho}\).
There exist \(\delta_1=\delta_1(n,\beta/\alpha)>0\) and \(\gamma=\gamma(n,\beta/\alpha)>0\)  such that, if
\begin{equation}\label{delta}
\Lambda\varrho+D_u(x,\varrho)+\inf_{\nu \in S^{n-1}} \mean{\partial E\cap B_{x,\varrho}}|\nu_E-\nu|^2\le \delta_1\,,
\end{equation}
then  \(\partial E\cap B_{\varrho/2}\) is a \(C^{1,\gamma}\) hypersurface. 
\end{theorem}
As shown in \cite[Section 5]{F}, 
Theorem \ref{eps0} implies that \(\mathcal H^{n-1}(\Sigma (E))=0\).   A useful  classical consequence of Theorem \ref{eps0} is the following lemma concerning convergence of minimizers:

\begin{lemma}\label{conv}Let \((u_k,E_k)\) be a sequence of \((\Lambda,1)\)-minimizers of \(\F^{1,\beta/\alpha}_1\) in \(B_1\) such that
\begin{equation}\label{treno}
\sup_{k\in \N} D_{u_k} (1)<\infty.
\end{equation}
Then, up to a subsequence, there exists \((u,E)\) a  \((\Lambda,1)\)-minimizer of \(\F^{1,\beta/\alpha}_1\) in \(B_{1}\) such that 
\[
\|u_k-u\|_{W^{1,2}(B_{3/4})}\to 0,\qquad |(E_k\Delta E)\cap B_{3/4}|\to 0\,.
\]
Moreover \(P(E_k, \cdot)\to P(E, \cdot)\) as Radon measures in \(B_{3/4}\), and  \(\partial E_k\cap \overline B_{3/4}\to \partial E\cap \overline B_{3/4}\) in the Kuratowski sense. Finally, if \(x_0\in  {\rm  Reg} ( E)\cap B_{1/2}\)
and \( \partial E_k\cap B_{3/4}\ni x_k\to x_0\), then  there exists a radius \(\overline \varrho>0\) (depending on \(E\) and \(x_0\)) such that, for \(k\) sufficiently large, \(\partial E_k\cap B_{x_k,\overline \varrho}\subset  {\rm  Reg} (E_k)\).
\end{lemma}

\begin{proof}The first  part of the statement concerning the strong $W^{1,2}$-convergence of \(u_k\) is classical, see for instance the proof of Theorem 4.1 in \cite{AB} (note that the sequence \((u_k, E_k)\) is precompact according to \eqref{treno} and \eqref{eq:upperdensity}). Also, Kuratowski convergence of \(\partial E_k\) to \(\partial E\) is an easy consequence of the  density estimates \eqref{density}.

Concerning the last part of the statement we start noticing that, by elliptic regularity, if \(x_0\in {\rm Reg} (E)\cap B_{1/2}\) then \(u\) is Lipschitz in a neighborhood  of \(x_0\) \cite[Theorem 5]{F}. In particular, taking into account the \(C^1\) regularity of \(\partial E\) at \(x_0\), there exits a radius \(\overline \varrho=\overline\varrho(x_0,\Lambda,\beta/\alpha)\) such that
\[
2\Lambda \overline \varrho+D_u(x_0,4\overline \varrho)+\inf_{\nu \in \boldsymbol S^{n-1}} \mean{\partial E\cap 2B_{x_0,4\overline \varrho}}|\nu_E-\nu|^2\le \delta_1/2\,,
\]
where \(\delta_1\) is the constant appearing in \eqref{delta}. By the strong convergence of the sequence \((u_k, E_k)\) and the convergence of \(x_k\) to \(x_0\) we immediately see that, for \(k\) large enough,
\[
2\Lambda\overline \varrho+D_u(x_k,2\overline \varrho)+\inf_{\nu \in \boldsymbol S^{n-1}} \mean{\partial E\cap 2B_{x_k,2 \overline\varrho }}|\nu_E-\nu|^2\le \delta_1\,.
\]
Theorem \ref{eps0} now implies that  \(\partial E_k\cap B_{x_k,\overline \varrho}\subset  {\rm  Reg} (E_k)\), as desired.
\end{proof}

\section{Proof of Theorems \ref{thm:main1} and \ref{thm:main2}}\label{3}
We now provide the proofs of Theorems  \ref{thm:main1} and \ref{thm:main2}.
\begin{proof}[Proof of Theorem \ref{thm:main2}]
Let \((u,E)\) be a \((\Lambda,r_0)\)-minimizer of \(\F_{\eps}^{\alpha,\beta}\) in \(B_1\). We aim to prove that \(\Sigma (E)\cap \overline B_{1/2}\) is \((\sigma,\hat \varrho)\)-porous in \(\partial E\), where \(\sigma=\sigma(n,\beta/\alpha)\) and \(\hat\varrho=\hat \varrho\big(n,r_0,\,\Lambda/\eps, \,\alpha D_u(1)/\eps\big)\). To this end let us set
\begin{equation}\label{defv}
v:=\sqrt{\alpha/\eps} \,\,u
\end{equation}
and note that, according to Section \ref{subsection:scaling}, \((v,E)\) is a \((\Lambda/\eps, r_0)\) minimizer of \(\F_1^{1,\beta/\alpha}\). Moreover
\begin{equation}\label{francoforte}
D_{v}(1)=\alpha D_u(1)\big/\eps .
\end{equation}
We now define
\begin{equation}\label{hatrho}
\hat{\varrho}:=\min\big\{1/2,r_0,\eps/\Lambda, \varrho_1(n,r_0,\Lambda, D_v(1)))\},
\end{equation}
where \(\varrho_1\) the  radius appearing in Lemma \ref{energybound}. Note that, according to \eqref{francoforte}, 
\[
\hat \varrho=\hat \varrho\big(n,r_0,\,\Lambda/\eps, \,\alpha D_u(1)/\eps \big)
\]
and that, by Lemma \ref{energybound},
\begin{equation}\label{jp}
D_{v}(x,\varrho)\le C_1\qquad \forall\, x\in B_{1/2} \,,\quad  \varrho\le \hat\varrho.
\end{equation}
Finally, for \(x\in  \partial E\cap B_{1/2}\) and \(\varrho\le \hat \varrho\) we set
\[
p(x,\varrho):=\frac{\sup\big\{r\in (0,\varrho): \textrm {there exists \(y\) s.t. } B_{y,r}\cap \partial E\subset \big(\partial E \setminus (\Sigma (E)\cap \overline B_{1/2})\big) \cap B_{x,\varrho}\big\}}{\varrho}\,,
\]
so that the conclusion of Theorem \ref{thm:main2} is equivalent to \(p(x,\varrho)\ge \sigma\) for some \(\sigma=\sigma(n,\beta/\alpha)\) (note that  \(\varrho\le \hat \varrho \le 1/2\) and \(x\in B_{1/2}\) imply \(B_{x,\varrho}\subset B_1\)). 

Let us argue by contradiction and  assume that  there exists a sequence  \((u_k,E_k)\) of \((\Lambda_k, r_{0,k})\)-minimizers  of \(\F_{\eps_k}^{\alpha_k, \beta_k}\) in \(B_1\), with \(\beta_k/\alpha_k=\beta/\alpha\), for which there exist a point \(x_k\in \partial E_k \cap B_{1/2}\) and a radius \(\varrho_k\le \hat \varrho_k\) (\(\hat \varrho_k\) as in \eqref{hatrho}) such  that
\[
p(x_k,\varrho_k)\to 0\qquad\textrm{as \(k\to \infty\)}.
\]
If we define \(v_k\) as in \eqref{defv} and set \(w_k:=v_k^{x_k, \varrho_k}\) and \(F_k:=E_k^{x_k,\varrho_k}\), we get a sequence  of \((1,1)\) minimizers of \(\F^{1,\beta/\alpha}\) in \(B_1\), with \(0\in \partial E\) and for which
\begin{equation}\label{pk}
p_k:=\sup\big\{r: \textrm {there exists \(y\) s.t. } B_{y,r}\cap \partial F_k\subset \big(\partial F_k \setminus \Sigma (F_k)\big)\cap B_1 \big\}\le p(x_k,\varrho_k)\to 0.
\end{equation}
According to \eqref{jp} and Section \ref{subsection:scaling}
\[
D_{w_k}(1)=D_{v_k}(x_k,\varrho_k)\le C_1\,,
\]
hence, thanks to Theorem \ref{eps0}, there exists a \((1,1)\) minimizer \((w_\infty, E_\infty)\) of \(\F_1^{1,\beta/\alpha}\) in \(B_1\)  such that 
\[
\|w_k-w_\infty\|_{W^{1,2}(B_{3/4})}\to 0\qquad |(F_k\Delta F_\infty) \cap B_{3/4}|\to 0\,,
\]
and \(0\in \partial F_\infty\). By Lemma \ref{lowerdensity}  \(\H^{n-1}(\partial F_\infty\cap B_{1/2})>0\)  and \(\H^{n-1}(\Sigma(\F_\infty))=0\),  hence there exists a regular point \(x_0 \in \partial F_\infty\cap B_{1/2}\). By  the Kuratowski convergence of \(\partial F_k\cap \overline B_{1/2}\) to \(\partial F_\infty\cap \overline B_{1/2}\), we can find  a sequence of points \(y_k\in  \partial F_k\cap \overline B_{1/2}\) such that \(y_k\to x_0\). According to  Theorem \ref{eps0} there exists a radius \(\overline \varrho>0\) such that, for \(k\) large, \(\partial F_k\cap B_{y_k,\overline \varrho}\subset  {\rm  Reg} (F_k)\). Since this last fact is in contradiction with \eqref{pk}, this concludes the proof of Theorem \ref{thm:main2}.
\end{proof}

Before proving Theorem \ref{thm:main1}, let us recall the following lemma concerning porous sets. Its proof can be obtained by following the same argument given in the Introduction of \cite{S}  or in  \cite[Lemma 5.10]{DS}. The key point is that a Alfhors regular set  (i.e., a set satisfying  \eqref{stime} below)  admits a ``dyadic cubes''  decomposition, as shown for instance in   \cite[Appendix]{D}. 

\begin{lemma}\label{poro} Let \(\Sigma\subset B_{1/2}\) be a  closed set and let \(K\subset B_1\) be a relatively closed set such that \(\Sigma\subset K\). Let  us assume that \(\Sigma\) is \((\sigma,\hat \varrho)\)-porous in \(K\) and that there exits a constant \(\hat C\) such that
\begin{equation}\label{stime}
\varrho^{n-1}/\hat C\le \H^{n-1}(K\cap B_{x,\varrho})\le \hat C\varrho^{n-1}\qquad \forall \, x\in B_{3/4}\quad \varrho\le \hat \varrho.
\end{equation}
Then there exists \(\kappa=\kappa(n,\sigma,\hat C)>0\) such that \(\dim_H \Sigma\le n-1-\kappa\).
\end{lemma}

Let us also remark that arguing as in \cite[Lemma 3.3] {DF} one can actually show the following stronger statement on the measure of the \(\varrho\)-neighborhood of \(\Sigma\), see Equations  (3.7) and (3.8) in \cite{DF}. 
This implies that the Minkowski dimension of $\Sigma$ is bounded by $n-1-\kappa$.

Note that (as it should) the constant \(\widetilde C\) below depends also on \(\hat \varrho\) while \(\kappa\) does not.

\begin{lemma}\label{poro2} Let \(\Sigma\) and \(K\) be as in Lemma \ref{poro}, then there exists constant \(\widetilde C=\widetilde C(n,\sigma, \hat C,\hat \varrho)>0\) and \(\kappa=\kappa(n,\sigma,\hat C)>0\) such that
\[
\big|\{x\in B_1: \dist(x,\Sigma)\le \varrho\}\big|\le \widetilde C \varrho^{1+\kappa}\qquad \forall \varrho\le 1/2
\]
\end{lemma}
We now prove Theorem \ref{thm:main1}.
\begin{proof}[Proof of Theorem \ref{thm:main1}] By Theorem \ref{thm:main2} \(\Sigma(E)\cap \overline B_{1/2}\) is \((\sigma,\hat \varrho)\)-porous in \(\partial E\cap B_1\) with \(\sigma=\sigma(n,\beta/\alpha)\) and \(\hat \varrho=\hat \varrho\big(n,r_0,\,\Lambda/\eps, \,\alpha D_u(1)/\eps\big)\).  Moreover, using \eqref{density} and arguing as in the proof of Theorem \ref{thm:main2}, we see  that for every \((\Lambda,r_0)\)-minimizers of \(\F^{\alpha,\beta}_\eps\)   it holds
\[
\varrho^{n-1}/C_2\le \mathcal H^{n-1}(\partial E\cap B_{x,\varrho})\le C_2\varrho^{n-1} \qquad \forall\, x\in B_{3/4}, \quad \varrho \le  \hat \varrho_1,
\]
where \(C_2=C_2(n,\beta/\alpha)\) and \(\hat \varrho_1=\hat \varrho_1\big(n,r_0,\beta/\alpha,\Lambda/\eps, \,\alpha D_u(1)/\eps\big)\). We can then apply Lemma \ref{poro} to deduce that there exists \(\kappa=\kappa(n,\beta/\alpha)>0\) such that 
\[
\dim_{\H} \bigl(\Sigma (E)\cap \overline B_{1/2}\bigr)\le n-1-\kappa.
\]
A simple scaling and covering argument then concludes the proof.
\end{proof}


\begin{thebibliography}{99}

\bibitem{AB} {\sc Ambrosio, L.; Buttazzo, G.}:  {\em An optimal design problem with perimeter penalization}. Calc. Var. Partial Differential Equations {\bf 1} (1993), 55--69.


\bibitem{D}{\sc David, G.}:  Wavelets and singular integrals on curves and surfaces. Lecture Notes in Mathematics, 1465. Springer-Verlag, Berlin, 1991.

\bibitem{Dms} {\sc David, G.}: {\em  $C^1$-arcs for minimizers of the Mumford-Shah functional.} SIAM J. Appl. Math. {\bf 56} (1996), 783--888


\bibitem{DS}{\sc David, G.;  Semmes, S.}: Fractured fractals and broken dreams. Self-similar geometry through metric and measure. Oxford Lecture Series in Mathematics and its Applications, 7. The Clarendon Press, Oxford University Press, New York, 1997.

\bibitem{DF} {\sc De Philippis, G.; Figalli, A.}: \emph{Higher integrability for minimizers of the Mumford-Shah functional}.   Arch. Ration. Mech. Anal. To appear.

\bibitem{EF}{\sc  Esposito, L.; Fusco, N.}: {\em  A remark on a free interface problem with volume constraint.}  J. Convex Anal. {\bf 18}  (2011),  417--426.

\bibitem{FJ}{\sc Fusco, N.; Julin, V.}: {\em
On the regularity of critical and minimal sets of a free interface problem.}
Preprint, 2014.


\bibitem{FK} {\sc Kohn, R.; Lin, F. H.}:  {\em Partial regularity for optimal design problems involving both bulk and surface energies.}  Chinese Ann. Math. Ser. B {\bf 20}  (1999), 137--158.

\bibitem{F} {\sc Lin, F.H.}:  {\em Variational problems with free interfaces} . Calc. Var. Partial Differential Equations {\bf 1} (1993), 55--69


\bibitem{KM}{\sc Kristensen, J.;  Mingione, G.}: {\em The singular set of Lipschitzian minima of multiple integrals.} Arch. Ration. Mech. Anal. {\bf 184} (2007), 341--369

\bibitem{MSms} {\sc Maddalena, F.; Solimini, S.}:  {\em Regularity properties of free discontinuity sets.} Ann. Inst. H. Poincar\'e Anal. Non Lin\'eaire {\bf 18} (2001), 675--685.

\bibitem{M}  {\sc Maggi, F.}: Sets of finite perimeter and geometric variational problems, volume 135 of Cambridge Studies
in Advanced Mathematics. Cambridge University Press, Cambridge, 2012. 

\bibitem{Rms}{\sc Rigot, S.}: {\em Big pieces of  $C^{1,\alpha}$-graphs for minimizers of the Mumford-Shah functional.} Ann. Scuola Norm. Sup. Pisa Cl. Sci. (4) {\bf 29} (2000), 329--349. 

\bibitem{S}{\sc Salli, A.}: {\em On the Minkowski dimension of strongly porous fractal sets in $\mathbb R^n$.} Proc. London Math. Soc. (3) {\bf 62 }(1991), 353--372.

\bibitem{Si} {\sc Simon, L.}:  Lectures on geometric measure theory, volume 3 of Proceedings of the Centre for Mathematical
Analysis. Australian National University, Centre for Mathematical Analysis, Canberra, 1983.

 \end{thebibliography}
\end{document}